\documentclass[12pt]{article}
\usepackage{color}
\usepackage[colorlinks=true]{hyperref}
\usepackage{amsmath,amssymb,amsthm,amscd,enumerate}
\usepackage{xcolor}
\usepackage{mathrsfs}
\usepackage[title]{appendix}
\numberwithin{equation}{section}

\newtheorem{theorem}{Theorem}[section]

\newtheorem{proposition}{Proposition}[section]
\newtheorem{lemma}{Lemma}[section]

\newtheorem{remark}{Remark}[section]

\usepackage{times}
\begin{document}
\title{From nonlinear Schr\"{o}dinger equation to interacting particle system: $1<p<2$}
\author{
Xin Liao\thanks{School of Mathematics and Statistics, Hunan Normal University, Changsha 410012, China. Email: \texttt{xin\_liao@whu.edu.cn}} 
\and
Juntao Lv\thanks{School of Mathematics and Statistics, South-Central Minzu University, Wuhan 430074, China. Email: \texttt{lvjuntao@whu.edu.cn}}}

\date{\today}
\maketitle
\begin{abstract}
  We investigate the limiting behavior of solutions with infinitely many peaks to nonlinear Schr\"{o}dinger equations
		\[-\varepsilon^{2} \Delta u_{\varepsilon}+u_{\varepsilon}=u_{\varepsilon}^{p}, \quad
		u_{\varepsilon}>0  \quad \text{in} \ \mathbb{R}^{n},
		\]
		as $\varepsilon\to 0$, where $p$ is Sobolev subcritical. We derive the interaction law among the limiting peak points and complete the analysis for the previously unresolved range  $1<p<2$,  extending the work of Ao, Lv, and Wang (J. Differential Equations, 2025).
\end{abstract}

\section{Introduction}
This work is a continuation of the study by Ao, Lv, and Wang \cite{awl}, extending their previous results to cover the full range of 
Sobolev subcritical  index.

\subsection{Background}

We investigate the limiting behavior of a class of positive solutions to the nonlinear Schr\"{o}dinger equations
\begin{equation}\label{eqn with epsilon}
		-\varepsilon^{2} \Delta u_{\varepsilon}+u_{\varepsilon}=u_{\varepsilon}^{p} \quad
		\text{in} \ \mathbb{R}^{n},
	\end{equation}
	as $\varepsilon\rightarrow 0$, where $1<p<\frac{(n+2)}{(n-2)_+}$ is Sobolev subcritical.
Equation (\ref{eqn with epsilon}) arises as the standing wave problem for the standard nonlinear Schr\"odinger equation and  the study of models in the biological theory of pattern formation \cite{T1952}, such as the Gray-Scott or Gierer-Meinhardt system \cite{GM1972, GS1983}.

By rescaling
	$u(x)=u_\varepsilon(\varepsilon x),$
the problem reduces to
	\begin{equation}\label{eqn}
		- \Delta u+u=u^p, \quad   u>0 \ \ \text{in} \ \mathbb{R}^{n},
	\end{equation}
which admits a ground state solution  $W$,  that is radially symmetric, decreasing, and decays to zero at infinity  \cite{BL1983}.
 Under the assumption that $u(x)\to 0$ as $|x|\to \infty$ or $u\in H^1(\mathbb{R}^{n})$,  the solution to \eqref{eqn} is unique up to translation, see \cite{GNN1981} and \cite{K1989}. However, much less is known about the solutions  that do not vanish at infinity.

 An important class of such solutions is the \emph{partially constrained solutions},
  vanishing at infinity in all but one direction, and we refer to them as solutions with a single
bump line.
   A simple example is the positive radial solution $W(x^\prime)$ in lower dimensions, which can be trivially extended to higher dimensions as $u(x^\prime,x_n)=W(x^\prime)$ .
The second one is the periodic Dancer solutions \cite{D2001}, which are nontrivially periodic in $x_n$ with a large period, these solutions have infinitely many bump points.

Besides partially constrained solutions, multiple-end solutions have been constructed. In particular, del Pino, Kowalczyk, Pacard, and Wei \cite{dKPW2010} used the Dancer solution to build solutions with multiple bump lines, whose locations are determined by a one-dimensional Toda system. Malchiodi \cite{M2009} constructed multiple-end solutions using the half Dancer solution, while Santra and Wei \cite{SW2013} obtained solutions combining one front with infinitely many bump points.

These solutions have natural counterparts in the geometry of constant mean curvature surfaces in $\mathbb{R}^3$: the radial solution corresponds to a sphere, the single bump line to a cylinder, the Dancer solution to a periodic Delaunay surface, and the multiple bump line solution to a  constant mean curvature surface with multiple Delaunay ends. Similarly, the coexistence of a front and bumps corresponds to a triunduloid.

 A natural question raised by Wei \cite{W2010} asks \emph{ whether all positive, partially constrained solutions are periodic in the remaining direction}; partial results exist \cite{GMX2011, PV2017, PVa2020, PV2020, PV2022}, but a complete answer to this question is still unknown.  For more background, we refer to \cite{awl} and the references therein.

\subsection{Main result}
In this paper, we focus on solutions with infinitely many bump points. Rather than addressing existence, we investigate the limiting positions of these bumps as $\varepsilon\to 0$, under natural assumptions mentioned in \cite{awl}.  More presicely, the solutions $u_\varepsilon$ satisfy:

\begin{description}
    \item[Peak structure] There exist infinitely many points $\{\xi_{\alpha,0}\} \subset \mathbb{R}^{n}$ such that
    \[
        \inf_{\alpha\neq\beta}|\xi_{\alpha,0}-\xi_{\beta,0}|>0,
    \]
    and a large constant $K^*$ with
    \[
        u_\varepsilon < 1/2 \quad \text{in } \mathbb{R}^{n} \setminus \bigcup_\alpha B_{K^*\varepsilon}(\xi_{\alpha,0}).
    \]
\end{description}

\begin{description}
    \item[Uniform condition] Defining
    \[
        \rho_\alpha := \min_{\beta\neq\alpha} |\xi_{\alpha,0}-\xi_{\beta,0}|,
    \]
    there exists $C_0\geq 1$ such that
    \[
        \frac{1}{C_0} \le \frac{\rho_\alpha}{\rho_\beta} \le C_0, \quad \text{for all } \alpha\neq\beta.
    \]
\end{description}

\medskip

Our main result is as follows:

\begin{theorem}\label{main result}
Under the above hypotheses, for any $\alpha\in\mathbb{Z}$, the  points $\{\xi_{\alpha,0}\}$ satisfy
\begin{equation}\label{limiting system}
    \sum_{\beta\neq\alpha} \ell_{\alpha\beta} \frac{\xi_{\beta,0}-\xi_{\alpha,0}}{|\xi_{\beta,0}-\xi_{\alpha,0}|} = 0,
\end{equation}
for some $\ell_{\alpha\beta}\ge 0$ satisfy
$
    \sum_{\beta\neq\alpha} \ell_{\alpha\beta} = 1,
$
and $\ell_{\alpha\beta}=0$ if $\beta\notin \mathcal{N}_\alpha$, where
\[
    \mathcal{N}_\alpha := \left\{\beta : |\xi_{\beta,0}-\xi_{\alpha,0}| = \min_{\gamma\neq \alpha} |\xi_{\gamma,0}-\xi_{\alpha,0}| \right\}.
\]
\end{theorem}

\begin{remark}
Equation \eqref{limiting system} describes an interacting particle system where each particle interacts with neighboring ones. Note that this is not strictly nearest-neighbor interaction, since $\ell_{\alpha\beta}$ may vanish even if $\beta \in \mathcal{N}_\alpha$.
\end{remark}

The case $p\ge 2$  was treated in \cite{awl}, and here we focus on the range  $1<p<2$. Our estimates can be adapted with minor modifications to cover $p\ge 2$.
For $p\geq 2$, the key idea in \cite{awl} for proving Theorem \ref{main result} via the reverse Lyapunov–Schmidt reduction method, inspired by the analysis in \cite{WW2019}, can be roughly summarized as follows:

\begin{enumerate}
  \item Construct an optimal approximation of the original solution $u_\varepsilon$  in the form
    \[u_\varepsilon(x)=\sum_{\alpha\in\mathbb{Z}}W\left(\frac{x-\xi_{\alpha,\varepsilon}}{\varepsilon}\right)(x)+\phi_\varepsilon(x),\]
    where $\phi_\varepsilon$ denotes the \emph{error function} satisfying an appropriate orthogonality condition, and
    $$\lim_{\varepsilon\to0}\frac{|\xi_{\alpha,\varepsilon}-\xi_{\alpha,0}|}{\varepsilon}=0,$$
    see Proposition \ref{prop orthogonal decomposition} below.
  \item Using the nondegeneracy of $W$ together with the orthogonality condition, derive an almost optimal pointwise estimate for the error function $\phi_\varepsilon$. In particular, the leading-order term of $\phi_\varepsilon$ arises from the interactions among different peaks.
  \item Analyze the equation satisfied by the error function $\phi_\varepsilon$ pointwise in each subdomain determined by $\xi_{\alpha,\varepsilon}$.  By performing a suitable translation and extracting a subsequence, they compute the leading-order term of the error equation in the limit $\varepsilon \to 0$,   obtaining a limiting equation. Testing this limiting equation against $\nabla W$ then  completes the proof.
\end{enumerate}

Our proof strategy is similar to that in \cite{awl}, the main difference is that, for $1<p<2$,
 the pointwise estimates from the second step that are too weak to yield a limiting equation. To overcome this difficulty,
we adopt the integral-estimate approach of \cite{liao}, testing the error equation itself rather than the limiting equation, and perform refined computations of the integral contributions from the linear term, the nonlinear term, and the interaction term. The detailed estimates are presented in Section 4.
\medskip

Return to partially constrained solutions, for  solutions of the form
\[
u_\varepsilon(x',x_n) = \sum_{\alpha\in\mathbb{Z}} W\Big(\frac{x-\xi_{\alpha,\varepsilon} e_n}{\varepsilon}\Big) + o(1)
\]
with $\xi_{\alpha,\varepsilon}\to \xi_\alpha$ as $\varepsilon\to 0$, and
\[
\cdots < \xi_{\alpha-1} < \xi_\alpha < \xi_{\alpha+1} < \cdots,
\]
an application of Theorem \ref{main result} yields
\[
\xi_\alpha = \frac{\xi_{\alpha-1} + \xi_{\alpha+1}}{2}, \quad \text{for all } \alpha\in\mathbb{Z},
\]
so that the limiting peak points are periodically distributed along the $x_n$-axis. Thus, our results give certain evidence in support of Wei's conjecture. 

\medskip

{\bf Notation:} In this paper, we use $C$  to denote universal constants independent of $\varepsilon$. They could vary from line to line. When comparing two quantities $A$ and $B$, if $A\leq CB$ for some universal constant $C$, we denote it as $A\lesssim B$.
	
	\section{Preliminaries}\label{sec preliminary}
	\setcounter{equation}{0}
\subsection{The property of ground state}
It is well known that for the ground state $W$, there exists a constant $C(n)>0$ such that
    \begin{equation}\label{infi}
      \left|W(x)- C(n)|x|^{-\frac{n-1}{2}} e^{-|x|}\right| \lesssim |x|^{-\frac{n+1}{2}} e^{-|x|}.
    \end{equation}
Moreover, for its gradient, one has $|\nabla W|\lesssim W$ and
    \begin{equation}\label{infi2}
    \left|\nabla W(x)- C(n)\frac{x}{|x|}|x|^{-\frac{n-1}{2}} e^{-|x|} \right| \lesssim |x|^{-\frac{n+1}{2}} e^{-|x|}.
    \end{equation}
As a consequence of \eqref{infi} and \eqref{infi2}, we have the following useful estimates:
\begin{lemma}\label{esti}
Let $y\in \mathbb{R}^n$ with $|y|>1$.  Then the following hold:
\begin{enumerate} 
  \item For all $x\in\mathbb{R}^n$,
  \begin{equation}\label{es1}
W(x) W(x+y)\lesssim |y|^{-\frac{(n-1)}{2}} e^{-|y|} .
\end{equation}
  \item For $ \alpha>\beta>0$,
\begin{equation}\label{es2}
\int_{\mathbb{R}^n} W^{\alpha}(x)  W^{\beta}(x+y)\,\mathrm{d}x \approx |y|^{-\frac{\beta(n-1)}{2}} e^{-\beta|y|}.
\end{equation}
  \item For $ p>1$, there exists a constant $\bar{c}>0$ such that
\begin{equation}\label{es3}
\left|\int_{\mathbb{R}^n} W^{p-1}(x)  \nabla W(x) W(x+y)\,\mathrm{d}x-\bar{c}\frac{y}{|y|}|y|^{-\frac{n-1}{2}}e^{-|y|} \right| \lesssim |y|^{-\frac{n+1}{2}}e^{-|y|}.
\end{equation}
\end{enumerate}

\end{lemma}
\begin{proof}
  See Section 3.1 in \cite{liao}.
\end{proof}

	Denote $Z_i:=\partial_{x_i}W$, $i=1,\dots, n$. These are solutions to the linearized equation
	\begin{equation}\label{linearized eqn around W}
		-\Delta \phi+\phi=pW^{p-1}\phi \quad \text{in} ~ \mathbb{R}^{n}.
	\end{equation}
	A well-known property of $W$ is its nondegeneracy (see \cite[ Lemma 13.4]{WW2014}), which is stated as follows.
	\begin{lemma}[Nondegeneracy]\label{nondegeneracy}
		If $\phi\in L^\infty(\mathbb{R}^{n})$ solves the linearized equation \eqref{linearized eqn around W}, then there exist constants $c_1$, $\dots$, $c_n$ such that
		\[\phi=\sum_{i=1}^{n}c_i Z_i.\]
	\end{lemma}

 The next proposition, which is \cite[Proposition 6.1]{AMPW2016}, is about the a priori estimate for the linearized problem. Given $\delta>0$, define the weighted norm as
	\[\|\phi\|_{\delta}:=\sup_{x\in \mathbb{R}^{n}}e^{\delta|x|}\left|\phi(x) \right|. \]
	
 \begin{proposition}[A priori estimate]\label{prop prior estimate}
		For any $f\in L^{\infty}(\mathbb{R}^{n})$ satisfying $\|f\|_\delta<\infty$, there exists a unique $\phi\in L^{\infty}(\mathbb{R}^{n})$ and $c\in \mathbb{R}^{n}$ solving
		\begin{equation}\label{eqn for linear operator}
			\Delta \phi-\phi+pW^{p-1}\phi+ c\cdot\nabla W =f \quad \text{in} \ \mathbb{R}^{n},
		\end{equation}
		and
		\[\int_{\mathbb{R}^{n}}\phi Z_i\,\mathrm{d}x=0, \quad i=1,\cdots,n.\]
Moreover,
		\[\|\phi\|_{\delta} +|c|\leq C\|f\|_{\delta}, \]
  where $C>0$ is a constant that does not depend on $f$.
	\end{proposition}

\subsection{Profile and  orthogonal condition}
Once we have performed a rescaling $u(x)=u_\varepsilon(\varepsilon x)$, the  points $\xi_{\alpha, 0}$ are scaled to $\xi_\alpha^\ast:=\frac{\xi_{\alpha, 0}}{\varepsilon}$. It is important to note that in these notations, there should be an $\varepsilon$-dependence, but we will not make it explicit the sake of simplicity.

For each $\alpha\in\mathbb{Z}$, we denote
	\begin{equation}\label{D_alpha}
A:=\{\xi_{\alpha}^*\},\quad
D_\alpha:=\min_{\beta\neq\alpha}|\xi_\alpha^\ast -\xi_\beta^\ast|=\frac{\rho_{\alpha}}{\varepsilon},\quad D:=\inf_{\alpha}D_\alpha . 
\end{equation}
We point out that as $$D\to \infty, \mbox{ as } \varepsilon \to 0.$$
In the following, we fix a cut-off function $\eta\in C_0^\infty(B_2)$, satisfying
	\begin{itemize}
		\item $0\leq \eta \leq 1$ and $\eta\equiv 1$ in $B_1$;
		\item $|\nabla\eta|+|\nabla^2\eta|\leq C$.
	\end{itemize}
	We also take a large constant $K$  , and denote
	\[\eta_K(x):=\eta\left(\frac{x}{K}\right).\]

The following results can be found in \cite{awl}.
	\begin{proposition}\label{prop orthogonal decomposition}
Under the hypotheses, we have the following decay estimate:
\begin{equation}\label{decay}
u(x)\leq Ce^{-\frac{d(x,A)}{C}}.\end{equation}
Moreover, for any fixed  $K>1000$  and $\varepsilon$ small enough,
 there exists a sequence of points $\xi_{\alpha}\in\mathbb{R}^{n}$  (with $\varepsilon$-dependence) such that $$|\xi_{\alpha}-\xi_\alpha^\ast|=o(1),$$ and for any $\alpha\in\mathbb{Z}$ and $i=1$, $\dots$, $n$,
		\begin{equation}\label{eq2.1}
			\int_{\mathbb{R}^{n}}\Big[u(x)-\sum_{\beta\in \mathbb{Z}}W(x-\xi_{\beta})\Big] \eta_K\left(x-\xi_\alpha\right)Z_{i}(x-\xi_{\alpha})\,\mathrm{d}x=0.
		\end{equation}
	\end{proposition}
\begin{remark}\label{relation of xi and xi*}
In Proposition \ref{prop orthogonal decomposition}, since we have $|\xi_\alpha-\xi_\alpha^\ast|=o(1)$, so
\begin{equation*}
D_\alpha=D_\alpha^\prime+o(1),
\end{equation*}
where
\begin{equation*}
D_\alpha^\prime:=\min_{\beta\neq\alpha}|\xi_\alpha-\xi_\beta|.
\end{equation*}
In the following, by  abusing notations, we will just use $D_\alpha$ instead of $D_\alpha^\prime$.
\end{remark}

\section{Estimate on the error function}
For $\xi_\alpha$ obtained by Proposition \ref{prop orthogonal decomposition}, we 
denote
	\[W_\alpha(x):=W(x-\xi_\alpha), \quad Z_{i,\alpha}(x):=Z_i(x-\xi_\alpha), \quad \eta_{K,\alpha}(x):=\eta_K(x-\xi_\alpha)\]
	and
	\begin{equation}\label{form of error fct}
		\phi:=u-\sum_{\alpha\in\mathbb{Z}}W_\alpha,
	\end{equation}
	which is the \emph{error function}. It satisfies the \emph{error equation}
	\begin{equation}\label{error eqn}
		-\Delta\phi+\phi= \Big(\sum_{\alpha\in\mathbb{Z}}W_\alpha+\phi\Big)^p-\sum_{\alpha\in\mathbb{Z}} W_\alpha^p
	\end{equation}
	and the \emph{orthogonal condition}
	\begin{equation}\label{orthogonal condition}
		\int_{\mathbb{R}^{n}}\phi \eta_{K,\alpha} Z_{i,\alpha}\,\mathrm{d}x=0, \quad \forall i=1,\cdots, n,  ~ \alpha\in\mathbb{Z}.
	\end{equation}

	Sometimes we also need another form of \eqref{error eqn},
	\begin{equation}\label{error eqn 2}
		-\Delta \phi+\phi=p\Big(\sum_{\alpha\in\mathbb{Z}}W_\alpha\Big)^{p-1}\phi+N(\phi) +\mathcal{I},
	\end{equation}
	where
	\[N(\phi):=\Big(\sum_{\alpha\in\mathbb{Z}}W_\alpha+\phi\Big)^p -\Big(\sum_{\alpha\in\mathbb{Z}}W_\alpha\Big)^p -p\Big(\sum_{\alpha\in\mathbb{Z}}W_\alpha\Big)^{p-1}\phi\]
	is the  nonlinear term, and
	\begin{equation}\label{interaction}
		\mathcal{I}:=\Big(\sum_{\alpha\in\mathbb{Z}}W_\alpha\Big)^p-\sum_{\alpha\in\mathbb{Z}}W_\alpha^p
		\end{equation}
	is the interaction term between different $W_\alpha$.

For any $\alpha\in\mathbb{Z}$, denote
\begin{equation}\label{Omega_alpha}
	\Omega_\alpha:=\{x: |x-\xi_\alpha| \leq |x-\xi_\beta|, ~ \forall \beta\neq\alpha\}
	\end{equation}
as the domain near $\xi_\alpha$.  It is easy to see that in $B_{\frac{D_{\alpha}}{2}}(\xi_{\alpha})\subset \Omega_\alpha$,  and 
    $$W_\alpha(x)\geq W_\beta(x),\quad\forall x\in \Omega_\alpha, \beta \neq \alpha.$$ In particular, we have:
\begin{lemma}The interaction term $\mathcal{I}$ satisfies the following estimate:

\begin{equation}\label{Holder estimate on interaction term}
\|\mathcal{I}\|_{C^{0}(\Omega_\alpha)}\lesssim \sum_{\beta\neq \alpha}  W_{\alpha}^{p-1} W_\beta +e^{-10C_0D}\lesssim
 e^{-\frac{pD_{\alpha}}{2}},  \quad \mbox{if } 1<p<2.                                          
\end{equation}
\end{lemma}
\begin{proof}
  Denote \begin{equation}\sigma_{\alpha}:= \sum_{\beta\neq \alpha} W_\beta.
  \end{equation}
  For a fix point $x\in\Omega_\alpha $, denote $R:= |x-\xi_{\alpha}|$, by the uniform condition, 
  $$\#\{\xi_{\beta}~|~\xi_{\beta}\in B_{2^{k+1}R}(x)\setminus  B_{2^{k}R}(x) \} =O\left(2^{kn}\Big(\frac{R}{D_{\alpha}}\Big)^n\right).$$
  If $R\geq 10C_0D_{\alpha}$,
   $$\sigma_{\alpha}\lesssim \sum_{k=1}^{\infty} 2^{kn}\Big(\frac{R}{D_{\alpha}}\Big)^n \left(2^{k}R\right)^{-\frac{n-1}{2}}  e^{-2^{k}R}  \lesssim e^{-10C_0D},$$
   which implies that
   $$  |\mathcal{I} |  \lesssim e^{-pD_{\alpha}}.$$
   For the case $R\leq 10C_0 D_{\alpha}$, we have
    $$\sigma_{\alpha}\lesssim \sum_{k=1}^{\infty} 2^{kn} \left(2^{k}R\right)^{-\frac{n-1}{2}}  e^{-2^{k}R}  \lesssim W_{\alpha},$$
 therefore, we obtain
  \begin{equation}\label{4.9}
  \begin{aligned}
   |\mathcal{I}|&\lesssim
   \left|(\sigma_{\alpha}+W_{\alpha})^p-W_{\alpha}^p\right|+\sum_{\beta\neq \alpha }W_{\alpha}^{p-1}W_{\beta}\\&
   \lesssim
    W_{\alpha}^{p-1}\sigma_{\alpha}\\&\lesssim
  \sum_{\beta\neq \alpha } e^{-(p-1) |x-\xi_{\alpha}|- |x-\xi_{\beta}| }\\&
\leq   \sum_{\beta\neq \alpha } e^{-(p-1)D_{\alpha}-(2-p) |x-\xi_{\beta}|   }  
   \\&\leq \sum_{\beta\neq \alpha }e^{-(p-1)D_{\alpha}- \frac{(2-p)|\xi_{\alpha}-\xi_{\beta}|}{2}   }
   \lesssim e^{-\frac{pD_{\alpha}}{2}}.
  \end{aligned}
  \end{equation}
 In the last inequality, we again use the uniform condition to deduce $$ \sum_{\beta\neq \alpha }e^{- \frac{(2-p)|\xi_{\alpha}-\xi_{\beta}|}{2}   }\lesssim e^{-\frac{(2-p)D_{\alpha}}{2} }.$$
This completes the proof.
\end{proof}

\begin{remark}
As seen in the proof of the above lemma,
  the uniform condition implies that, when considering the interaction between a given ground state $W_{\alpha}$ and infinitely many other ground states, it suffices to account only for the finitely many
$W_{\beta}$ whose centers $\xi_{\beta}$ are nearest to $\xi_{\alpha}$, since the contributions from the remaining terms are negligibly small.
\end{remark}

	In the following of this section, we establish a $C^{1}$ estimate for the error function $\phi$. It will be seen that the main order term in this estimate comes from the interaction between different peaks.

Our proof strategy is similar to \cite{awl}.
	We will consider the estimate on $\phi$ separately in the inner domains $B_K(\xi_\alpha)$ (for each $\alpha\in\mathbb{Z}$) and the outer domain $\mathbb{R}^{n}\setminus \cup_{\alpha\in\mathbb{Z}}B_{L}(\xi_\alpha)$. Here $K$ is the constant used in the definition of orthogonal decomposition in Section 2, while $L$ is a constant satisfying $1\ll L \ll K$. Notice that since $L<K$, the inner and outer domains have nonempty overlaps.
	
	\subsection{Outer problem} The error function $\phi$ satisfies
	\begin{equation}\label{outer eqn}
		-\Delta \phi+\phi =E_{out},
	\end{equation}
	where
\[E_{out} :=\Big(\sum_{\alpha\in\mathbb{Z}} W_\alpha+\phi\Big)^p-\sum_{\alpha\in\mathbb{Z}} W_\alpha^p.\]

	We have the following pointwise bound on $E_{out}$.
	\begin{lemma}\label{estimate for Eout}
In the outer domain $\Omega_{out}:=\mathbb{R}^{n}\setminus \cup_{\alpha\in\mathbb{Z}}B_{L}(\xi_\alpha)$,
	there exists a small constant $\tau >0$ (depending  only on $L$) such that
 \[ |E_{out}|\leq \tau |\phi|+\mathcal{I}.\]
	\end{lemma}
\begin{proof}
In fact, by \eqref{decay}, in the outer domain we have
$$\sum_{\alpha\in\mathbb{Z}} W_\alpha \mbox{ and } \phi =o_{L}(1),$$
and the lemma then follows directly from the definition of $E_{out}$.
\end{proof}

Based on this lemma and Kato's inequality, we see that in the outer domain,
\begin{equation}
    -\Delta|\phi|+(1-\tau)|\phi| \leq \mathcal{I}.
    \end{equation}

 Denote the Green function of the elliptic operator $-\Delta+(1-\tau)$ by $G_\tau (\cdot,\cdot)$. By the asymptotic behavior of the Green function, we have the following upper bound
	\begin{equation}\label{Green fct bound}
		0< G_\tau (x,y) \leq Ce^{-(1-\tau)|x-y|} \quad \mbox{ for }|x-y|\geq 1.
	\end{equation}
Moreover, we have $G\geq 0$ and
  $$(1-\tau)\int_{\mathbb{R}^n} G(x,y)\mathrm{d}y=1. $$
By the maximum principle and \eqref{4.9}, we obtain
\begin{align}\label{Green representation}
		|\phi(x)|&\leq \sum_{\beta\in\mathbb{Z}} \int_{\partial B_L(\xi_\beta)}P_{\tau,\beta} (x,y)|\phi(y)| \,\mathrm{d}y\\
  &+C\sum_{\beta\in\mathbb{Z}} \int_{\Omega_\beta}G_\tau (x,y)W_\beta^{p-1}(y) \sum_{\gamma\neq \beta} W_\gamma(y)\,\mathrm{d}y, \nonumber
\end{align}
where $P_{\tau,\beta} (\cdot,\cdot)$ is the Poisson kernel associated to $-\Delta+(1-\tau )$ in the exterior domain
$\mathbb{R}^{n}\setminus B_{L}(\xi_\beta)$.
For any fixed $\alpha\in\mathbb{Z}$, denote 
$$ \widehat{\Omega_\alpha}:=\left\{x~|~ d(x, \Omega_\alpha)<1\right\},$$
 if $x\in  \widehat{\Omega_\alpha} \setminus B_{K/3}(\xi_\alpha)$, then plugging the exponential decay estimates on $G_\tau$ and $P_\tau$ into \eqref{Green representation} and using   \eqref{Holder estimate on interaction term} gives
\begin{align}\label{outer estimate 1}
|\phi(x)| \lesssim & e^{-\frac{(1-\tau)K}{3}}\|\phi\|_{L^\infty\left(\partial B_L(\xi_\alpha)\right)} +\sum_{\beta\neq\alpha} e^{-(1-\tau)|x-\xi_\beta| }\|\phi\|_{L^\infty\left(\partial B_L(\xi_\beta)\right)} \nonumber\\
+&e^{-10C_0D}+ \int_{\Omega_\alpha} G_\tau (x,y) W_\alpha^{p-1}(y) \sum_{\gamma\neq \alpha} W_\gamma(y)\,\mathrm{d}y\\
+& \sum_{\beta\neq\alpha} \int_{\Omega_\beta}G_\tau (x,y)W_\beta^{p-1}(y) \sum_{\gamma\neq \beta} W_\gamma(y)\,\mathrm{d}y. \nonumber
\end{align}

Here, the first two terms on the right-hand side are obtained from the boundary integrals as described in \eqref{Green representation}, and an exponential upper bound on the Poisson kernel similar to \eqref{Green fct bound} has been utilized.

The last two integrals in \eqref{outer estimate 1} satisfy the following estimates:
\begin{equation}
   \int_{\Omega_\alpha} G_\tau (x,y) W_\alpha^{p-1}(y) \sum_{\gamma\neq \alpha} W_\gamma(y)\,\mathrm{d}y \leq Ce^{-\frac{pD_{\alpha}}{2}} \int_{\Omega_\alpha} G_\tau (x,y)  \,\mathrm{d}y
   \leq Ce^{-\frac{pD_{\alpha}}{2}}.
\end{equation}
For $y\in \Omega_\beta$, using the uniform condition,  we have
$$\sum_{\gamma\neq \beta} e^{-|y-\xi_\gamma|}\lesssim e^{-\frac{D_{\beta}}{2}}\lesssim e^{-\frac{D}{2}}.$$
Fix $p\in (1,2)$. By choosing $L$ sufficiently large, we can ensure that $0<\tau<\min\left\{2-p, \frac{p-1}{C_0}\right\}$. Then, for $x\in \widehat{\Omega_\alpha} \setminus B_{K/3}(\xi_\alpha)$, we obtain
\begin{equation}
\begin{aligned}
&\sum_{\beta\neq\alpha} \int_{\Omega_\beta}G_\tau (x,y)W_\beta^{p-1}(y) \sum_{\gamma\neq \beta} W_\gamma(y)\,\mathrm{d}y\\
\lesssim &\sum_{\beta\neq\alpha}  \int_{\Omega_\beta}G_\tau (x,y)W_\beta(y) \sum_{\gamma\neq \beta} W^{p-1}_\gamma(y)\,\mathrm{d}y\\
\lesssim &\sum_{\beta\neq\alpha}  \int_{\Omega_\beta} \sum_{\gamma\neq \beta}e^{-(1-\tau)|y-x|-|y-\xi_\beta|-(p-1)|y-\xi_\gamma|}\,\mathrm{d}y\\&+W_\beta(x) \sum_{\gamma\neq \beta} W^{p-1}_\gamma(x)\int_{B_{1}(x)}G(x,y)\,\mathrm{d}y\\
\lesssim &\sum_{\beta\neq\alpha} e^{-(1-\tau)|x-\xi_\beta|-\frac{p-1}{2}D} \int_{\Omega_\beta} e^{-\tau |y-\xi_\beta|}\,\mathrm{d}y+e^{-\frac{p}{2}D_{\alpha}}\\
\lesssim &e^{-\frac{(1-\tau)}{2}D_{\alpha}-\frac{p-1}{2}D}.
\end{aligned}
\end{equation}
Therefore, we conclude that 
\begin{align}\label{outer estimate 2}
\|\phi\|_{L^\infty\left(\widehat{\Omega_\alpha}\setminus _{K/3}(\xi_\alpha)\right)} &\lesssim   e^{-\frac{(1-\tau)K}{3}}\|\phi\|_{L^\infty\left(\partial B_L(\xi_\alpha)\right)} +\sum_{\beta\neq\alpha} e^{-(1-\tau)|x-\xi_\beta| }\|\phi\|_{L^\infty\left(\partial B_L(\xi_\beta)\right)}\nonumber\\
&+e^{-\frac{(1-\tau)}{2}D_{\alpha}-\frac{p-1}{2}D} .
\end{align}

\subsection{Inner problem}

In view of \eqref{orthogonal condition}, for each $\alpha\in\mathbb{Z}$, we define \emph{the $\alpha$-th inner function} to be $\phi_\alpha:=\phi \eta_{K,\alpha}$. It satisfies
\begin{equation}
	-\Delta \phi_\alpha+\phi_\alpha=pW_\alpha^{p-1} \phi_\alpha+E_\alpha,
\end{equation}
where
\begin{equation}\label{eqn for E alpha}
    E_\alpha := \left[\Big(\sum_{\beta\in\mathbb{Z}}W_\beta+\phi\Big)^p-\sum_{\beta\in\mathbb{Z}} W_\beta^p -pW_\alpha^{p-1}\phi\right]\eta_{K,\alpha}-\Delta\eta_{K,\alpha}\phi-2\nabla\phi\nabla\eta_{K,\alpha} .
\end{equation}
Furthermore, by \eqref{orthogonal condition}, $\phi_\alpha$ satisfies the orthogonal condition
\begin{equation}
    \int_{\mathbb{R}^{n}}\phi_\alpha Z_{i,\alpha}\,\mathrm{d}x=0, \quad \text{for } i=1, \dots, n.
\end{equation}
Define
$$\|\phi\|_{\delta,\alpha}:=\sup_{x\in \mathbb{R}^{n}}e^{\delta|x-\xi_{\alpha}|}\left|\phi(x) \right|.$$
By Proposition \ref{prop prior estimate} (for any $\delta>0$ sufficiently small, which will be determined below),
\begin{equation}\label{estimate on delta norm of inner fct}
    \|\phi_\alpha\|_{\delta,\alpha} \lesssim \|E_\alpha\|_{\delta,\alpha}.
\end{equation}
Therefore, an explicit estimate of $E_\alpha$ is needed.
\begin{lemma}
For any fixed $\alpha\in\mathbb{Z}$, it holds that
\begin{equation}\label{estimate on delta norm of inner fct, 2}
\|\phi_\alpha\|_{\delta,\alpha} \lesssim e^{\delta K} D_\alpha^{-\frac{n-1}{2}}e^{-D_\alpha }+ e^{\delta K}\|\phi\|_{C^{1}\left(B_{K}(\xi_\alpha)\setminus B_{K/2}(\xi_\alpha)\right)}.
\end{equation}
\end{lemma}
\begin{proof}
The support of $E_\alpha$ lies within $B_K(\xi_\alpha)$. In this ball, $W_\alpha$ has a positive lower bound (dependent on $K$ but fixed), while
\[ \sum_{\beta\neq\alpha}W_\beta\leq  e^{-D_{\alpha}}, \quad |\phi|\ll 1.\]
Consequently, in this ball, via the uniform condition, we get
\begin{equation}
\begin{aligned}
& \left|\Big(\sum_{\beta\in\mathbb{Z}}W_\beta+\phi\Big)^p-\sum_{\beta\in\mathbb{Z}} W_\beta^p -pW_\alpha^{p-1}\phi\right|\\
  \lesssim  &pW_\alpha^{p-1}\sum_{\beta\neq\alpha}W_{\beta}+O\bigg(\Big(\sum_{\beta\neq \alpha}W_{\beta}\Big)^2+\phi^2\bigg) + e^{-pD_{\alpha}}
|\phi|^p\\
\lesssim  & D_\alpha^{-\frac{n-1}{2}}e^{-D_{\alpha}}+o\left(\|\phi_\alpha\|_{L^\infty(\mathbb{R}^{n})}\right).
\end{aligned}
\end{equation}

Plugging the estimate from this lemma into \eqref{estimate on delta norm of inner fct}, we get
\[\|\phi_\alpha\|_{\delta,\alpha} \lesssim e^{\delta K} D_\alpha^{-\frac{n-1}{2}}e^{-D_\alpha }+ e^{\delta K}\|\phi\|_{C^{1}\left(B_{K}(\xi_\alpha)\setminus B_{K/2}(\xi_\alpha)\right)}+o\left(\|\phi_\alpha\|_{L^\infty(\mathbb{R}^{n})}\right).\]
Since $\|\phi_\alpha\|_{L^\infty(\mathbb{R}^{n})} \leq \|\phi_\alpha\|_{\delta,\alpha}$, this implies that
\begin{equation}
\|\phi_\alpha\|_{\delta,\alpha} \lesssim e^{\delta K} D_\alpha^{-\frac{n-1}{2}}e^{-D_\alpha }+ e^{\delta K}\|\phi\|_{C^{1}\left(B_{K}(\xi_\alpha)\setminus B_{K/2}(\xi_\alpha)\right)}.
\end{equation}
\end{proof}
By the definition of $\eta_{K,\alpha}$ and because $3L<K/2$, we have $\phi\equiv \phi_\alpha$ in $B_{3L}(\xi_\alpha)$. Consequently, \eqref{estimate on delta norm of inner fct, 2} implies that
\[\|\phi\|_{L^\infty\left(B_{3L}(\xi_\alpha)\right)}\lesssim e^{\delta K}D_\alpha^{-\frac{n-1}{2}} e^{-D_\alpha }+ e^{\delta K}\|\phi\|_{C^{2,1/2}\left(B_{K}(\xi_\alpha)\setminus B_{K/2}(\xi_\alpha)\right)}.\]
Applying  interior Calderón–Zygmund estimation to \eqref{error eqn 2} and using the Sobolev embedding, we deduce that
\begin{equation}\label{inner estimate}
  \|\phi\|_{C^{1}\left(B_{2L}(\xi_\alpha)\right)} \lesssim e^{\delta K}D_\alpha^{-\frac{n-1}{2}}e^{-D_\alpha }+ e^{\delta K}\|\phi\|_{C^{1}\left(B_{K}(\xi_\alpha)\setminus B_{K/2}(\xi_\alpha)\right)}.
\end{equation}

\subsection{Completion of the estimate on the error function}
 \begin{proposition}\label{prop estimate on error fct}
 For  the subcritical exponent
 $p\in (1,2)$,
there exists a  constant $\tau\in \left(0,\min\left\{2-p, \frac{p-1}{C_0}\right\}\right)$, such that
    \begin{equation}\label{estimate on error fct}
    \|\phi\|_{C^1(\Omega_{\alpha})} \lesssim  
                                                      e^{-\frac{(1-\tau)D_{\alpha}}{2}-\frac{(p-1)D}{2}}.
    \end{equation}
 \end{proposition}
 \begin{remark}
   As we will see in the proof of Theorem \ref{main result}, the only estimate we need is $$\|\phi\|_{C^1\left(B_{\frac{D_{\alpha}}{2}}(\xi_{\alpha})\right)}=o\left(D_{\alpha}^{-\frac{n-1}{2}}e^{-\frac{D_{\alpha}}{2}}\right),$$
  which is guaranteed by \eqref{estimate on error fct}, since $\tau \in \left(0, \frac{p-1}{C_0}\right)$.
 \end{remark}
 \begin{proof}
   Using \eqref{outer estimate 2} and \eqref{inner estimate}, together with the fact
   $$ \sum_{\beta\neq\alpha} e^{-(1-\tau)|x-\xi_\beta| }\lesssim e^{-\frac{(1-\tau)}{2}D_{\alpha}},\quad \forall x\in \widehat{\Omega_\alpha}\setminus _{K/3}(\xi_\alpha),  $$
    we obtain 
   \begin{equation}\label{return}
   \begin{aligned}
\|\phi\|_{L^\infty\left(\widehat{\Omega_\alpha}\setminus _{K/3}(\xi_\alpha)\right)} &\lesssim   e^{-\frac{(1-\tau)K}{3}}\|\phi\|_{L^\infty\left(\partial B_L(\xi_\alpha)\right)} +e^{-\frac{(1-\tau)}{2}D_{\alpha}} \sup_{\beta\neq\alpha}\|\phi\|_{L^\infty\left(\partial B_L(\xi_\beta)\right)}\\
&+e^{-\frac{(1-\tau)}{2}D_{\alpha}-\frac{p-1}{2}D}\\&
 \lesssim  e^{-\frac{(1-\tau)K}{3}+\delta K}\|\phi\|_{C^{1}\left(B_{K}(\xi_\alpha)\setminus B_{K/2}(\xi_\alpha)\right)}+e^{-\frac{(1-\tau)}{2}D_{\alpha}}.
\end{aligned}
\end{equation}
Applying Calderón–Zygmund estimation  and Sobolev embedding once more, we further obtain
$$  \|\phi\|_{C^1\left(\Omega_\alpha\setminus B_{K/2}(\xi_\alpha)\right)} \lesssim   e^{-\frac{(1-\tau)K}{3}+\delta K}\|\phi\|_{C^{1}\left(B_{K}(\xi_\alpha)\setminus B_{K/2}(\xi_\alpha)\right)}+  e^{-\frac{(1-\tau)}{2}D_{\alpha}}.$$
By choosing $\delta<\frac{(1-\tau)}{3}$ and taking $K$ sufficiently large, we then deduce
$$  \|\phi\|_{C^1\left(\Omega_\alpha\setminus B_{K/2}(\xi_\alpha)\right)} \lesssim    e^{-\frac{(1-\tau)}{2}D_{\alpha}}.$$
Returning to \eqref{estimate on delta norm of inner fct, 2}, it follows that
$$  \|\phi\|_{C^1\left( B_{K/2}(\xi_\alpha)\right)}  \lesssim  \|\phi\|_{L^{\infty}\left( B_{K}(\xi_\alpha)\right)}\lesssim   e^{-\frac{(1-\tau)}{2}D_{\alpha}}. $$
 Since $\alpha$ is arbitrary, we conclude
$$\|\phi\|_{L^\infty\left(\mathbb{R}^n\right)} \lesssim  e^{-\frac{(1-\tau)}{2}D}.$$
Substituting this back into \eqref{return} and using $0<\tau<2-p$, we obtain 
$$  \|\phi\|_{L^\infty\left(\widehat{\Omega_\alpha}\setminus B_{K/3}(\xi_\alpha)\right)} \lesssim   e^{-\frac{(1-\tau)K}{3}+\delta K}\|\phi\|_{C^{1}\left(B_{K}(\xi_\alpha)\setminus B_{K/2}(\xi_\alpha)\right)}+ e^{-\frac{(1-\tau)}{2}D_{\alpha}-\frac{(p-1)D}{2}}.$$
Repeating this argument, the proof is complete.
 \end{proof}
\section{Proof of Theorem \ref{main result}}
Denote 
$$
\sigma = \sum_{\alpha\in \mathbb{Z}} W_\alpha, \quad \kappa_{\alpha}:=D_{\alpha}^{-\frac{n-1}{2}} e^{-D_{\alpha}} 
$$
For each $\alpha$,
by direct calculation we recall that
\begin{equation}\label{3-13}
\begin{aligned}
(-\Delta+1) \phi-pW_{\alpha}^{p-1}\phi
&= \mathcal{I} +  N(\phi)+ p\sigma^{p-1} \phi-pW_{\alpha}^{p-1} \phi ,
\end{aligned}
\end{equation}
where
  \begin{align*}
  N(\phi):=|\sigma+\phi|^{p-1}(\sigma+\phi) -\sigma^p -p\sigma^{p-1} \phi,\quad \mathcal{I}:=\sigma^p -\sum\limits_{\alpha\in \mathbb{Z}}  W_\alpha^p.
  \end{align*}

The proof of Theorem \ref{main result} relies on the following three estimates of  interaction term:
\begin{enumerate}
 \item By \eqref{es1} and the uniform condition,  $$\sigma_{\alpha}W_{\alpha}\lesssim \kappa_{\alpha} \mbox{~in } \mathbb{R}^n.$$
  \item By \eqref{es2}  and the uniform condition, for \(a,b>0\) with \(a\neq b\), one has
\[
\sum_{\beta \neq \alpha}\int_{\mathbb{R}^n} W_{\beta}^{a}\, W_{\alpha}^{b}\, \,\mathrm{d}x \;\sim\; \kappa_{\alpha}^{\min\{a,b\}} .
\]
 
  \item 
The last one follows from   \eqref{es3} together with  Lemma 2.1 in \cite{Wei}:
\begin{lemma}\label{lemma1}
  \begin{equation}
  \begin{aligned}
\int_{\mathbb{R}^n}  \mathcal{I} \nabla W_{\alpha} \,\mathrm{d}x&=p \sum_{\beta \neq \alpha}\int_{\mathbb{R}^n} W_\alpha^{p-1}W_\beta \nabla W_{\alpha} \,\mathrm{d}x
+o(\kappa_{\alpha})\\&
= p\bar{c}\sum_{\beta\neq \alpha} \frac{\xi_{\beta}-\xi_{\alpha}}{|\xi_{\beta}-\xi_{\alpha}|}|\xi_{\beta}-\xi_{\alpha}|^{-\frac{n-1}{2}}e^{-|\xi_{\beta}-\xi_{\alpha}|}+o(\kappa_{\alpha}).
\end{aligned}
\end{equation}
\end{lemma}
\begin{proof}
 Let's decompose $$\int_{\mathbb{R}^n}   \mathcal{I} \nabla W_{\alpha} \, \,\mathrm{d}x=J_1+J_2+J_3+J_4,$$
  where
  \[ J_1=\int_{W_{\alpha}>\sigma_{\alpha}}\Big(pW_{\alpha}^{p-1}\sigma_{\alpha}-\sum_{\beta \neq \alpha}W_\beta^p\Big) \nabla W_{\alpha}  \,\mathrm{d}x ,  \]
  \[J_2=\int_{W_{\alpha}>\sigma_{\alpha}}  \Big( \sigma^{p}-pW_{\alpha}^{p-1}\sigma_{\alpha}-W_{\alpha}^p \Big)  \nabla W_{\alpha}  \,\mathrm{d}x,\]
  \[ J_3=\int_{W_{\alpha}\leq\sigma_{\alpha}}  \Big( p\sigma_{\alpha}^{p-1}W_\alpha+\sigma_{\alpha}^{p}-\sum_{\alpha\in \mathbb{Z}}W_\alpha^p \Big)\nabla W_{\alpha}  \,\mathrm{d}x ,\]
  \[J_4= \int_{W_{\alpha}\leq\sigma_{\alpha}}  \left( \sigma^{p}-\sigma_{\alpha}^{p}-p\sigma_{\alpha}^{p-1}W_\alpha \right)\nabla W_{\alpha}  \,\mathrm{d}x .\]
For some $\theta>0$ small enough, we have
\begin{equation}\label{est-J1}
\begin{aligned}
&\left| J_1 - p \int W_{\alpha}^{p-1}\sigma_{\alpha} \nabla W_{\alpha} \,\mathrm{d}x\right|
\\&=\left| \int_{W_{\alpha}>\sigma_{\alpha}} \sum_{\beta\neq \alpha} W_\beta^p \nabla W_{\alpha} \, \mathrm{d}x+ \int_{W_{\alpha} \leq \sigma_{\alpha}} p W_{\alpha}^{p-1} \sigma_{\alpha} \nabla W_{\alpha} \, \mathrm{d}x \right| \\&
\lesssim \sum_{\beta\neq \alpha} \int \left( W_\beta^{p-\theta} W_{\alpha}^{1+\theta} + W_\beta^{1+\theta} W_{\alpha}^{p-\theta}  \right)\, \mathrm{d}x \\&
\lesssim o( \kappa_{\alpha}) .
\end{aligned}
\end{equation}

Next, for $J_2$, it holds that
\begin{equation}\label{est-J2}
\begin{aligned}
|J_2|
&\lesssim \int_{W_{\alpha} > \sigma_{\alpha}} W_{\alpha}^{p-1} \sigma_{\alpha}^2  \, \mathrm{d}x
\lesssim \int W_{\alpha}^{p-\theta} \sigma_{\alpha}^{1+\theta}  \, \mathrm{d}x
\lesssim o( \kappa_{\alpha}) .
\end{aligned}
\end{equation}

Similarly, for $J_3$, we estimate
\begin{equation}\label{est-J3}
\begin{aligned}
|J_3|
&\lesssim \int \sigma_{\alpha}^{p-\theta} W_{\alpha}^{1+\theta}  \, \mathrm{d}x
+ \int_{W_{\alpha} \leq \sigma_{\alpha}} \Big( \sigma_{\alpha}^p \, \mathrm{d}x- \sum_{\beta\neq \alpha} W_\beta^p \Big) \nabla W_{\alpha}  \, \mathrm{d}x\\
&\lesssim o( \kappa_{\alpha}) + \sum_{\substack{\beta,\gamma\\ \beta,\gamma \neq \alpha,\, \beta\neq \gamma}} \int W_\beta^{p-1} W_\gamma W_{\alpha} \, \mathrm{d}x \\
&\lesssim o( \kappa_{\alpha}).
\end{aligned}
\end{equation}
Finally, for $J_4$, we obtain
\begin{equation}\label{est-J4}
\begin{aligned}
|J_4|
&\lesssim \int_{W_{\alpha} \leq \sigma_{\alpha}} W_{\alpha}^3 \sigma_{\alpha}^{p-2}  \, \mathrm{d}x
\lesssim \int W_{\alpha}^{1+\theta} \sigma_{\alpha}^{p-\theta}  \, \mathrm{d}x
\lesssim o( \kappa_{\alpha}) .
\end{aligned}
\end{equation}

Combining \eqref{est-J1}–\eqref{est-J4} and \eqref{es3}, we complete the proof.
\end{proof}

\end{enumerate}

Now we can give the complete proof of Theorem \ref{main result}:

\begin{proof}[Proof of Theorem \ref{main result}] 
Let $1<p<2$.
  Test the equation \eqref{3-13} with $\nabla W_{\alpha}$  on $B_{\frac{D_{\alpha}}{2}}( \xi_{\alpha})$   for some fixed $\alpha$,  by Proposition \ref{prop estimate on error fct}, we deduce that
  \begin{equation}\label{1111}
  \begin{aligned}
  \int_{B_{\frac{D_{\alpha}}{2}}( \xi_{\alpha})} & \left( \mathcal{I} +  N(\phi)+ p\sigma^{p-1} \phi-pW_{\alpha}^{p-1} \phi \right)\nabla W_{\alpha} \, \mathrm{d}x\\
  =&\int_{B_{\frac{D_{\alpha}}{2}}( \xi_{\alpha})}  \left[(-\Delta+1) \phi-pW_{\alpha}^{p-1}\phi\right]\nabla W_{\alpha} \, \mathrm{d}x\\
  =&\int_{B_{\frac{D_{\alpha}}{2}}( \xi_{\alpha})}  \left[(-\Delta+1) \nabla W_{\alpha}-pW_{\alpha}^{p-1}\nabla W_{\alpha}\right]\phi \, \mathrm{d}x\\
  &+\int_{\partial B_{\frac{D_{\alpha}}{2}}( \xi_{\alpha})} \left(  \phi\partial_n \nabla W_{\alpha} -\partial_n\phi\nabla W_{\alpha}\right) \,\mathrm{d}x\\
  =&\int_{\partial B_{\frac{D_{\alpha}}{2}}( \xi_{\alpha})} (\phi\partial_n \nabla W_{\alpha} -\partial_n\phi\nabla W_{\alpha}) \,\mathrm{d}x\\
  \lesssim &D_{\alpha}^{\frac{n-1}{2}} e^{-\frac{(2-\tau)D_{\alpha}}{2}-\frac{(p-1)D}{2}}=o(\kappa_{\alpha}).
  \end{aligned}
  \end{equation}
We point out that in the last inequality, we have used the fact  $\tau<\frac{p-1}{C_0}$.

  Note that
  \begin{equation}
  \begin{aligned}
\Bigg|\int_{B^c_{\frac{D_{\alpha}}{2}}( \xi_{\alpha})}  \mathcal{I}\nabla W_{\alpha} \,\mathrm{d}x\Bigg|&\lesssim |J_3|+|J_4|+\Bigg|\int_{\Omega_{\alpha}\setminus B_{\frac{D_{\alpha}}{2}}( \xi_{\alpha}) } \mathcal{I}\nabla W_{\alpha} \,\mathrm{d}x\Bigg|\\&
\lesssim o(\kappa_{\alpha})+ \int_{B^c_{\frac{D_{\alpha}}{2}}( \xi_{\alpha}) }  (W_{\alpha}^p+e^{-10C_0D})W_{\alpha}
\\&= o(\kappa_{\alpha}) .
    \end{aligned}
  \end{equation}
  Hence by Lemma \ref{lemma1},  we  obtain $$\int_{B_{\frac{D_{\alpha}}{2}}( \xi_{\alpha})}  \mathcal{I}\nabla W_{\alpha} \,\mathrm{d}x=p\bar{c}\sum_{\beta\neq \alpha} \frac{\xi_{\beta}-\xi_{\alpha}}{|\xi_{\beta}-\xi_{\alpha}|}|\xi_{\beta}-\xi_{\alpha}|^{-\frac{n-1}{2}}e^{-|\xi_{\beta}-\xi_{\alpha}|}+o(\kappa_{\alpha}).$$
 On the other hand, since
  \begin{equation}
  \begin{aligned}
   (\sigma^{p-1}-W_{\alpha}^{p-1}) \phi |\nabla W_{\alpha}|& \lesssim
  \begin{cases}
 |\phi| W_{\alpha}^{p-1}\sigma_{\alpha}, & \mbox{if } W_{\alpha}>\sigma_{\alpha} \\
   \sigma_{\alpha}^{p-1} |\phi| W_{\alpha}, & \mbox{if } W_{\alpha}\leq \sigma_{\alpha} .
  \end{cases}\\&
  \lesssim
    \begin{cases}
  \sigma_{\alpha}W_{\alpha}|\phi|^{p-1}  , & \mbox{if } |\phi|< W_{\alpha}, W_{\alpha}>\sigma_{\alpha}  \\
    |\phi|^{p+1}, & \mbox{if } |\phi|\geq W_{\alpha}>\sigma_{\alpha}\\
     |\phi|^{p+1}, & \mbox{if } W_{\alpha}\leq \sigma_{\alpha}\leq |\phi| \\
    \sigma_{\alpha} |\phi|^{p-1} W_{\alpha}, & \mbox{otherwise}  ,
  \end{cases}
  \end{aligned}
  \end{equation}
and $e^{-\frac{(2-\tau)D_{\alpha}}{2}-\frac{(p-1)D}{2}}=o\left(e^{-\frac{D_{\alpha}}{2}}\right)$, we have
\begin{equation}\label{4.111}
\begin{aligned}
\Bigg{|}\int_{B_{\frac{D_{\alpha}}{2}}( \xi_{\alpha})} \left(\sigma^{p-1}-W_{\alpha}^{p-1}\right) \phi \nabla W_{\alpha} \,\mathrm{d}x\Bigg{|}&\lesssim
 \int_{B_{\frac{D_{\alpha}}{2}}( \xi_{\alpha})} \left( \sigma_{\alpha} |\phi|^{p-1} W_{\alpha}+|\phi|^{p+1} \right) \,\mathrm{d}x\\&
 \lesssim   D_{\alpha}^n\kappa_{\alpha}e^{-\frac{(p-1)D_{\alpha}}{2}}+D_{\alpha}^ne^{-\frac{(p+1)D_{\alpha}}{2}}\\&
 \lesssim o(\kappa_{\alpha}).
 \end{aligned}
\end{equation}
Finally, for the nonlinear term, we have
   \begin{equation}\label{3.36}
   \begin{aligned}
    &\Bigg|\int_{B_{\frac{D_{\alpha}}{2}}( \xi_{\alpha})}  N(\phi)\nabla W_{\alpha} \,\mathrm{d}x\Bigg|\\
    \lesssim &  \int_{B_{\frac{D_{\alpha}}{2}}( \xi_{\alpha})} \Big{|}|\sigma+\phi|^{p-1}(\sigma+\phi) -\sigma^p -p\sigma^{p-1} \phi\Big{|}W_{\alpha} \,\mathrm{d}x\\
    \lesssim &\int_{\{|\phi|>\sigma\} \cap B_{\frac{D_{\alpha}}{2}}( \xi_{\alpha}) } |\phi|^{p+1} \,\mathrm{d}x + \int_{\{|\phi|\leq \sigma\}\cap B_{\frac{D_{\alpha}}{2}}( \xi_{\alpha})}   \sigma^{p-2} |\phi|^2W_{\alpha} \,\mathrm{d}x\\
    \lesssim  &D_{\alpha}^ne^{-\frac{(p+1)D_{\alpha}}{2}}
    +\int_{ B_{\frac{D_{\alpha}}{2}}( \xi_{\alpha})}   \sigma |\phi|^{p-1}W_{\alpha} \,\mathrm{d}x\\
     \lesssim &o(\kappa_{\alpha}) .
    \end{aligned}
    \end{equation}
  Therefore, we deduce from \eqref{1111}-\eqref{3.36} that 
  \begin{equation}
  \frac{\sum_{\beta\neq \alpha} \frac{\xi_{\beta}-\xi_{\alpha}}{|\xi_{\beta}-\xi_{\alpha}|}|\xi_{\beta}-\xi_{\alpha}|^{-\frac{n-1}{2}}e^{-|\xi_{\beta}-\xi_{\alpha}|} }{\kappa_{\alpha}}=o(1).
    \end{equation} 
     Since $\sum_{\beta\in \mathcal{N}_\alpha} |\xi_{\alpha}-\xi_{\beta}|^{-\frac{n-1}{2}}e^{-|\xi_{\alpha}-\xi_{\beta}|}\sim \kappa_{\alpha}$,  while  $$\sum_{\beta \neq \alpha, \beta\notin \mathcal{N}_\alpha} |\xi_{\alpha}-\xi_{\beta}|^{-\frac{n-1}{2}}e^{-|\xi_{\alpha}-\xi_{\beta}|}=o(\kappa_{\alpha}),$$
    by choosing a subsequence $\varepsilon_{n}\to 0$ such that 
    $$\frac{|\xi_{\beta,n}-\xi_{\alpha,n}|^{-\frac{n-1}{2}}e^{-|\xi_{\beta,n}-\xi_{\alpha,n}|} }{\sum_{\beta\in \mathcal{N}_\alpha} |\xi_{\alpha,n}-\xi_{\beta,n}|^{-\frac{n-1}{2}}e^{-|\xi_{\alpha,n}-\xi_{\beta,n}|}} \to l_{\alpha\beta}  $$
    and $$\frac{\xi_{\beta,n}-\xi_{\alpha,n}}{|\xi_{\beta,n}-\xi_{\alpha,n}|}\to \frac{\xi_{\beta,0}-\xi_{\alpha,0}}{|\xi_{\beta,0}-\xi_{\alpha,0}|},$$
     we finish the proof.
\end{proof}

\end{document}